\documentclass[12pt]{article}
\usepackage{amsmath}
\usepackage{amsfonts}
\usepackage{amsthm}
\usepackage{epsfig}

\def\ZZ{{\mathbb Z}}
\def\Z{{\mathcal Z}}

\def\Ker{{\rm Ker}}
\newtheorem{theorem}{Theorem}
\newtheorem{lemma}[theorem]{Lemma}

\begin{document}
\title{Markov bases of binary graph models of $K_4$-minor free graphs}
\author{Daniel Kr{\'a}l'\thanks{Institute for Theoretical Computer Science (ITI), Faculty of Mathematics and Physics, Charles University, Malostransk\'e n\'am\v est\'\i{} 25, Prague, Czech Republic. E-mail: {\tt kral@kam.mff.cuni.cz}. Institute for Theoretical Computer Science is supported by the Ministry of Education of the Czech Republic as project 1M0545.}\and
        Serguei Norine\thanks{Department of Mathematics, Princeton University, Princeton, NJ 08540-1000. E-mail: {\tt snorin@math.princeton.edu}. The author was supported in part by NSF under Grant No. DMS-0701033.}\and
        Ond{\v r}ej Pangr{\'a}c\thanks{Department of Applied Mathematics, Faculty of Mathematics and Physics, Charles University, Malostransk\'e n\'am\v est\'\i{} 25, Prague, Czech Republic. E-mail: {\tt pangrac@kam.mff.cuni.cz}}}
\date{}
\maketitle

\begin{abstract}
Markov width of a graph is a graph invariant defined as the maximum
degree of a Markov basis element for the corresponding graph model for binary contingency
tables. We show that a graph has Markov
width at most four if and only if it contains no $K_4$ as a minor, answering a question of Develin and Sullivant.
We also present a lower bound of order $\Omega(n^{2-\varepsilon})$
on the Markov width of $K_n$.
\end{abstract}

\section{Introduction}

A \emph{contingency table} is a $d_1 \times d_2 \times \cdots \times d_n$ array of non-negative integers. Contingency tables are used to record and analyze the relationship between $n$ discrete random variables $X_1, X_2, \ldots, X_n$, where $X_k$ takes values in the set $[d_k]=\{1,2,\ldots,d_k\}$ for $k=1,2,\ldots, n$. In \emph{hierarchial models} a simplicial complex $\Delta$ on $\{1,\ldots,n\}$ encodes interactions between the variables. One can estimate how well  empirical data in the given table fits a hierarchial model by comparing statistics of this table with statistics of a random non-negative integral table with the same set of marginals. In~\cite{bib-diaconis98+} an algebraic approach for generating such a random table has been presented. This approach can be informally summarized as follows. A finite set of \emph{moves}, such that any two tables with same set of marginals are connected by a sequence of such moves, is computed. Such a set of moves is called a \emph{Markov basis}. Given a Markov basis, a random table is generated by performing a random walk using the moves in the basis.

Thus, description of Markov bases of a given model is of interest and has attracted attention of researchers in recent years. For a more detailed introduction see~\cite{bib-develin03+,bib-hosten07+}. In this paper, following~\cite{bib-develin03+}, we concentrate our attention on binary graph models, i.e. hierarchial models of  $2\times\cdots\times 2$-contingency tables, where the simplicial complex $\Delta$, which encodes the variable interactions, is a graph.

Let us now give a formal definition of a Markov basis of a binary graph model. For a finite set $X=\{x_1,x_2, \ldots, x_n\}$ let $\Z(X) \simeq \ZZ^{2^n}$ denote the integral lattice with an orthonormal basis $e_{a(x_1)a(x_2)\ldots a(x_n)}$ indexed by binary labeling $a: X \to \{0,1\}$ of $X$. For a subset $Y \subset X$ there exists a natural projection map $\pi_{X \to Y}: \Z(X) \to \Z(Y)$ defined as the linear extension of the mapping $e_{a} \to e_{a|_Y}$. In most cases  the original set $X$ will be understood from context and we will frequently abbreviate $\pi_{X \to Y}$ to $\pi_Y$. Similarly, we will abbreviate $\pi_{\{a,b\}}$ to $\pi_{ab}$.

Given a graph $G$ let $$\pi_G:\Z(V(G))\to \bigotimes_{e \in E(G)}\Z(V(e))$$ be a linear mapping obtained as the product of the maps $\pi_{ij}$, $ij\in E$. The mappings $\pi_{ij}$ correspond to $2$-way marginals of a $2\times\cdots\times 2$ table.
A finite subset $B\subseteq\Ker(\pi_G)$ is a {\em Markov basis} if for every non-negative integral vectors $v_1,v_2\in\Z(V(G))$
with $\pi_G(v_1)=\pi_G(v_2)$, there exist $u_1,\ldots,u_\ell\in\pm B$
satisfying the following:
$$v_1+\sum_{k=1}^{\ell}u_k=v_2\qquad\mbox{and}\qquad
  v_1+\sum_{k=1}^{\ell'}u_k \;\mbox{ is non-negative for all $1\le\ell'\le\ell$.}$$
The \emph{Markov width} $\mu(G)$ of $G$ is the smallest integer $k$
such that there exists a Markov base $B$ of $G$ with $||v||_1\le 2k$
for every $v\in B$ where $||v||_1$ is the $\ell_1$-norm of $v$. The motivation for considering Markov width as the measure of complexity of the set of Markov bases of a binary graph model comes from the fact that a binary model can be alternatively defined as a binomial ideal. In this setting Markov bases correspond to generating sets of the ideal, and Markov width equals to the degree of the largest minimal
generator of the ideal. We omit the precise definitions, as we do not make use of this reinterpretation in our arguments, and refer the reader to~\cite{bib-develin03+,bib-sturmfels+} for details.

It is known~\cite{bib-dobra,bib-geiger+,bib-takken} that $\mu(G) = 2$ if and only $G$ is a forest, and that $\mu(G) \geq 4$, otherwise. In this paper we characterize graphs with $\mu(G)=4$, answering a question of Develin and Sullivant~\cite{bib-develin03+}. We show that $\mu(G) \leq 4$ if and only if $G$ does not contain a subdivision of the complete graph $K_4$ as a subgraph. Additionally, we investigate the Markov width of complete graphs.
Develin and Sullivant~\cite{bib-develin03+} showed that the Markov width $\mu(K_n)$ of the complete graph on $n$ vertices is lower bounded
by $2n-2$. We strengthen this lower bound, showing that $\mu(K_n)\ge\Omega(n^{2-\varepsilon})$ for every $\varepsilon>0$.

\section{Markov width of $K_4$-free graphs}

We start by describing a standard construction used in inductive arguments on $K_4$-minor free graphs.
{\em Series-parallel graphs} are graphs with two distinguished vertices called {\em poles}, obtained from elementary graphs
by a recursive construction.
The simplest series-parallel graph is
an edge $uv$ with the two poles being its end-vertices.
If $G_1$ and $G_2$ are series-parallel graphs with poles
$u_1$ and $v_1$, and $u_2$ and $v_2$, respectively, then
the graph $G$ obtained by identifying the vertex $v_1$ with $u_2$
is also a series-parallel graph and its two poles are the vertices
$u_1$ and $v_2$. The graph $G$ obtained in this way is called
the {\em serial join} of $G_1$ and $G_2$. The {\em parallel join}
of $G_1$ and $G_2$ is the graph obtained by identifying the
vertex $u_1$ with $u_2$ and the vertex $v_1$ with $v_2$;
the poles are the identified vertices. The {\em series-parallel graphs}
are precisely those that can be obtained from edges by a series
of serial and parallel joins. The sequence of such joins
leading to a construction of a graph $G$
is called a {\em series-parallel decomposition} of $G$.
The series-parallel decomposition of $G$ is not unique.

In our considerations, we will need the following (folklore) lemma.

\begin{lemma}
\label{lm-sp}
Every $2$-connected $K_4$-minor free graph is a series-parallel graph.
If $G$ is a $2$-connected $K_4$-minor free graph that is not a cycle,
then there exists a series-parallel decomposition of $G$ such that $G$
is obtained by a parallel join of at least three series-parallel graphs.
\end{lemma}

\noindent Observe that if the last operation in a series-parallel decomposition of $G$
is a parallel join, then $G$ is $2$-connected (the converse is also true). In particular,
we can apply Lemma~\ref{lm-sp} to such a graph $G$.

Finally, following standard graph theory terminology, we define a {\em $\{u,v\}$-bridge} $B$ of $G$
to be a connected subgraph of $G$ such that either $E(B)=\{uv\}$ or
for some component $C$ of $G\setminus \{u,v\}$ the set $E(B)$ consists of all edges of $G$ with at least one
end in $V(C)$. Note that a series parallel-graph $G$ with poles $u$ and $v$ can be obtained from the set of its $\{u,v\}$-bridges by a sequence of parallel joins.

In this section we characterize graphs with Markov width equal to four.
The cycles have this property, as shown in~\cite{bib-develin03+}.

\begin{lemma}[\cite{bib-develin03+}]
\label{lm-cycle}
If $G$ is a cycle then $\mu(G)=4$.
\end{lemma}

Throughout the proof of the main theorem of this section,
we will use the following straightforward observations repeatedly, and so we state them as lemmas.
For convenience, let $\Z(\emptyset)=\ZZ$ and $\pi_{X \to \emptyset}(z)=||z||_1$ for $z \in \Z(X)$.

\begin{lemma}
\label{lm-glue-cutsame}
Let $X_1$ and $X_2$ be finite sets, let $X = X_1\cup X_2$ and $Y=X_1\cap X_2$.
Let $z\in\Z(X)$ and $\overline{z}\in\Z(X_1)$ such that $\pi_{X\to Y}(z)=\pi_{X_1\to Y}(\overline{z})$.
Then there exists $z'\in\Z(X)$ such that $\pi_{X\to X_1}(z')=\overline{z}$,
$\pi_{X\to X_2}(z')=\pi_{X\to X_2}(z)$ and $||z-z'||_1=||\pi_{X\to X_1}(z)-\overline{z}||_1$.
\end{lemma}

\begin{lemma}
\label{lm-glue-swaps}
Let $X_1$ and $X_2$ be finite sets and let $X = X_1 \cup X_2$.
Let $z,z' \in \Z(X)$ be such that $\pi_{X \to X_i}(z)=\pi_{X \to X_i}(z')$ for $i=1,2$.
Then there exist vectors $z=z_0,z_1,z_2,\ldots,z_{\ell}=z' \in \Z(X)$ such that $||z_k-z_{k-1}||_1 = 4$ and $\pi_{X \to X_i}(z_{k-1})=\pi_{X \to X_i}(z_k)$
for $i=1,2$ and $k=1,\ldots,\ell$.
Moreover, if $z$ and $z'$ are non-negative, then $z_0,z_1,\ldots,z_{\ell}$ can be chosen to be non-negative.
\end{lemma}

\begin{lemma}
\label{lm-glue-cutchange}
Let $X_1$ and $X_2$ be finite sets, let $X = X_1\cup X_2$ and $Y=X_1\cap X_2$.
Let $z_1,z'_1\in\Z(X_1)$ and $z_2,z'_2\in\Z(X_2)$ such that $\pi_{X_1\to Y}(z_1)=\pi_{X_2\to Y}(z_2)$ and
$\pi_{X_1\to Y}(z'_1)=\pi_{X_2\to Y}(z'_2)$.
If $$||\pi_{X_1\to Y}(z_1)-\pi_{X_1\to Y}(z'_1)||_1=||z_1-z'_1||_1=||z_2-z'_2||_1\;\mbox{,}$$
then there exist $z,z'\in\Z(X)$ such that $\pi_{X\to X_i}(z)=z_i$, $\pi_{X\to X_i}(z')=z'_i$ and
$||z-z'||_1=||z_1-z'_1||_1$.
\end{lemma}

We first show that the Markov width of every series-parallel graph is at most four. In fact we prove a slightly stronger and more technical result.

\begin{theorem}
\label{thm-sp}
Let $G$ be a series-parallel graph with a vertex set $X$ and poles $u$ and $v$. If $z,z'\in\Z(X)$ are
two non-negative vectors with $\pi_G(z)=\pi_G(z')$, then there exist non-negative vectors $z_0,\ldots,z_{\lambda}\in\Z(X)$
such that
\begin{enumerate}
\item $z_0=z$, $z_\lambda=z'$,
\item $\pi_G(z)=\pi_G(z_k)$ for every $k=0,\ldots,\lambda$,
\item $||z_k-z_{k-1}||_1\le 8$ for every $k=1,\ldots,\lambda$, and
\item if $\pi_{uv}(z_{k-1})\not=\pi_{uv}(z_k)$, then $||z_k-z_{k-1}||_1=4$.
\end{enumerate}
\end{theorem}

\begin{proof}
The proof proceeds by induction on the order of $G$, i.e., $|X|$.
If $|X|=2$, then the graph $G$ is an edge and thus $z=z'$. The claim
readily follows.

Assume now that $G$ is a graph of order at least three obtained from graphs $G_1$ and $G_2$ by a serial or
a parallel join. Let $X_i$ be the vertex set of $G_i$, $i=1,2$. We distinguish four cases:
\begin{itemize}
\item {\bf $G$ is obtained by a serial join of $G_1$ and $G_2$.}\\
      Let $w$ be the vertex shared by $G_1$ and $G_2$; by symmetry, we can assume that the pole $u$ is contained in $G_1$ and $v$ in $G_2$.
      We apply induction to $G_i$ with $z^i=\pi_{X_i}(z)$ and $z'^i=\pi_{X_i}(z')$, for $i=1,2$.
      Let, for $i=1,2$, $z^i_0,\ldots,z^i_\ell$ be the resulting sequences of vectors (note that by padding the sequences
      with $\pi_{X_i}(z')$ at the end, we can assume that the sequences have the same length).

      We construct the required sequence $z_0,\ldots,z_{\lambda}$ as follows. We start by constructing a sequence of vectors
      We have now constructed a sequence of vectors $z_0,\ldots,z_{\ell}$ such that $z_0=z$,
      $\pi_{X_1}(z_k)=\pi_{X_1}(z^1_k)$ and $\pi_{X_2}(z_k)=\pi_{X_2}(z)$ for $k=1,2,\ldots, \ell$, and Properties 2,3 and 4 from the lemma statement are satisfied. Let $k\in [1,\ell]$ and assume the vectors $z_0,\ldots,z_{k-1}$ have been defined.
      If $\pi_{uw}(z^1_{k-1})=\pi_{uw}(z^1_k)$, apply Lemma~\ref{lm-glue-cutsame} with $X'_1=X_1$, $X'_2=X_2\cup\{u\}$,
      $z=z_{k-1}$ and $\overline{z}=z^1_k$ and set $z_k$ to be the resulting vector $z'$.
      Clearly, $||z_k-z_{k-1}||_1=||z^1_{k}-\pi_{X_1}(z_{k-1})||_1=||z^1_k-z^1_{k-1}||_1\le 8$ and
      $\pi_{uvw}(z_{k-1})=\pi_{uvw}(z_k)$. Since $\pi_{G_1}(z^1_{k-1})=\pi_{G_1}(z^1_k)$,
      it follows $\pi_G(z_{k-1})=\pi_G(z_k)$. In particular, Properties 2, 3 and 4 are satisfied in this step.

      If $\pi_{uw}(z^1_{k-1})\not=\pi_{uw}(z^1_k)$, then $||z^1_k-z^1_{k-1}||_1=4$ by Property 4.
      By Lemma~\ref{lm-glue-cutsame} applied with $z=z_{k-1}$ and $\overline{z}=z^1_k$, there exists
      $z_k=z'$ such that $\pi_{X_1}(z_k)=z^1_k$, $\pi_{X_2}(z_k)=\pi_{X_2}(z^2_{k-1})=\pi_{X_2}(z)$ and
      $||z_k-z_{k-1}||_1=4$. Again, $\pi_{G_1}(z^1_{k-1})=\pi_{G_1}(z^1_k)$ implies that
      $\pi_G(z_{k-1})=\pi_G(z_k)$.
      Since $||z_k-z_{k-1}||_1=4$, Properties 2, 3 and 4 are also satisfied  in this step.

      We have now constructed a sequence $z_0,\ldots,z_{\ell}$ such that $z_0=z$,
      $\pi_{X_1}(z_{\ell})=\pi_{X_1}(z')$ and $\pi_{X_2}(z_{\ell})=\pi_{X_2}(z)$. An analogous
      argument yields the existence of a sequence $z_{\ell},\ldots,z_{2\ell}$ such that
      $\pi_{X_1}(z_{2\ell})=\pi_{X_1}(z')$, $\pi_{X_2}(z_{2\ell})=\pi_{X_2}(z')$ satisfying Properties 2, 3 and 4.
      By Lemma~\ref{lm-glue-swaps}, the sequence $z_0,\ldots,z_{2\ell}$ can be completed
      to a sequence $z_0,\ldots,z_{\lambda}$ such that $z_{\lambda}=z'$ and
      $||z_k-z_{k-1}||_1=4$ for $k=\ell+1,\ldots,\lambda$. Clearly, the resulting sequence
      has Property 2. Since the $\ell_1$-norm of the vectors $z_k-z_{k-1}$, $k>2\ell$, is four,
      the sequence $z_{2\ell},\ldots,z_{\lambda}$ has also Properties 3 and 4.
\item {\bf $G$ is obtained by a parallel join of $G_1$ and $G_2$, $uv$ is an edge of $G$, and $G$ has at least three $\{u,v\}$-bridges.}\\
      By permuting the order of the parallel joins in the series-parallel decomposition of $G$,
      we can assume that neither $G_1$ nor $G_2$ is an edge. By symmetry, $G_1$ contains the edge $uv$.
      Let $G'_1$ be $G_1$ and $G'_2$ be $G_2$ with the edge $uv$ added.

      We apply induction to $G'_i$, $z^i=\pi_{X_i}(z)$ and $z'^i=\pi_{X_i}(z')$, $i=1,2$.
      Let $z^i_0,\ldots,z^i_\ell$ be the resulting sequence of vectors (note that by padding the sequences
      with $\pi_{X_i}(z')$ at the end, we can assume the sequences to have the same length).
      By Lemma~\ref{lm-glue-cutsame}, there exist a sequence of vectors $z_0,\ldots,z_{2\ell}$ such that
      $\pi_{X_1}(z_i)=z^1_i$ and $\pi_{X_2}(z_i)=z^2_0$ for $i=0,\ldots,\ell$,
      $\pi_{X_1}(z_i)=z^1_\ell$ and $\pi_{X_2}(z_i)=z^2_{i-\ell}$ for $i=\ell+1,\ldots,2\ell$ and
      $||z_i-z_{i-1}||\le 8$ for $i=1,\ldots,2\ell$. Clearly, this sequence has Properties 2 and 3.

      Lemma~\ref{lm-glue-swaps} yields that there is a sequence $z_{2\ell},\ldots,z_\lambda$ such that
      $z_\lambda=z'$, $\pi_{X_i}(z_{2\ell})=\cdots=\pi_{X_i}(z')$ and $||z_k-z_{k-1}||_1\le 4$
      for $k=2\ell+1,\ldots,\lambda$.
      The presence of the edge $uv$ in $G$ implies that $\pi_{uv}(z_0)=\cdots=\pi_{uv}(z_{\lambda})$
      which yields Property 4.
\item {\bf $G$ is obtained by a parallel join of $G_1$ and $G_2$ and $uv$ is not an edge of $G$.}\\
      If $\pi_{uv}(z)=\pi_{uv}(z')$, add the edge $uv$ to $G$ and proceed as in the previous case.
      Hence, $\pi_{uv}(z)\not=\pi_{uv}(z')$. Observe that there exists a (unique) sequence of vectors
      $w_0,\ldots,w_m\in\ZZ(\{u,v\})$ such that
      $w_0=\pi_{uv}(z)$, $w_m=\pi_{uv}(z')$, $||w_r-w_{r-1}||_1=4$
      for $r=1,\ldots,m$ and $||\pi_{uv}(z')-\pi_{uv}(z)||_1=4m$.

      We now apply induction for $G_i$, $z^i=\pi_{X_i}(z)$ and $z'^i=\pi_{X_i}(z')$, $i=1,2$.
      Let $z^i_0,\ldots,z^i_{\ell_i}$ be the resulting sequence of vectors. By Property 4,
      there exist indices $k^i_r$ such that $\pi_{uv}(z^i_{k^i_r-1})=w_{r-1}$ and
      $\pi_{uv}(z^i_{k^i_r})=w_r$ for $r=1,\ldots,m$ and $i=1,2$. Since it is possible to prolong the sequence
      by repeating some of the vectors several times, we can assume that $k^1_r=k^2_r$;
      let $k_r$ be their common value in the rest.

      By Lemma~\ref{lm-glue-cutchange}, there exist vectors $z_1,\ldots,z_m$ and $z'_0,\ldots,z'_{m-1}$
      such that $\pi_{X_i}(z_r)=z^i_{k_r}$
      for $i=1,2$ and $r=1,\ldots,m$,
      $\pi_{X_i}(z'_{r-1})=z^i_{k_r-1}$ for $i=1,2$ and $r=1,\ldots,m$, and
      $||z_r-z'_{r-1}||_1=4$ for every $r=1,\ldots,m$.
      For convenience, set $z_0=z$ and $z'_m=z'$.

      Let $G'$ be the graph obtained from $G$ by adding the edge $uv$.
      The choice of the indices $k_r$ implies that $\pi_{G'}(z_r)=\pi_{G'}(z'_r)$ for $r=0,\ldots,m$.
      In particular, we can apply the argument used
      in the previous case for $G'$ with $z_0$ and $z'_0$, $G'$ with $z_1$ and $z'_1$, \dots,
      $G'$ with $z_r$ and $z'_r$ and concatenate the obtained sequences of vectors. Clearly,
      the final sequence has Properties 1, 2 and 3.
      Because of the presence of the edge $uv$ in $G'$, the mapping
      $\pi_{uv}$ is constant inside each of the $r$ sequences in the concatenation.
      Since $||z_r-z'_{r-1}||_1=4$ for $r=1,\ldots,m$, the resulting sequence also has Property 4.
\item {\bf $G$ is obtained by a parallel join of $G_1$ and $G_2$, $uv$ is an edge of $G$ and $G$ has only two $\{u,v\}$-bridges.}\\
      Clearly, one of the $\{u,v\}$-bridges is the edge $uv$.
      If $G$ is a cycle, then there exist vectors $z_0,\ldots,z_{\lambda}$ satisfying Properties 1, 2 and 3
      by Lemma~\ref{lm-cycle}. Since $uv$ is an edge, $\pi_{uv}(z_0)=\cdots=\pi_{uv}(z_\ell)$
      which implies that the sequence also satisfies Property 4.

      On the other hand, if $G$ is not a cycle,
      then Lemma~\ref{lm-sp} implies that $G$ has another series parallel decomposition, say with poles $u'$ and $v'$,
      such that $G$ has at least three $\{u',v'\}$-bridges. Based on whether $G$ contains the edge $u'v'$,
      we apply the arguments presented in the second case or the third case to obtain a sequence
      of vectors $z_0,\ldots,z_\lambda$ satisfying Properties 1, 2, 3 and 4
      with respect to $u'$ and $v'$.
      Since $uv$ is an edge of $G$ and thus $\pi_{uv}(z_0)=\cdots=\pi_{uv}(z_\lambda)$,
      Property 4 also holds with respect to $u$ and $v$.
\end{itemize}
\end{proof}

It is now easy to derive the main result of this section:

\begin{theorem}
\label{thm-main}
The Markov width $\mu(G)$ of a graph $G$ is at most four if and only if
$G$ does not contain $K_4$ as a minor. In particular, $\mu(G)=4$ if and only if
$G$ is not a forest and has no $K_4$ as a minor.
\end{theorem}

\begin{proof}
It is shown in~\cite{bib-develin03+} that $\mu(K_4)=6$, and $\mu(H)\le\mu(H')$ if $H$ is obtained from $H'$ by a sequence of edge contractions and vertex deletions. Therrefore, the Markov width $\mu(G)$ of every graph $G$ containing $K_4$ as a minor
is at least six. Moreover, if $G$ is not a forest, then $\mu(G) \geq 4$
by Lemma~\ref{lm-cycle}.

On the other hand, it is easy to see that, the Markov width of a graph is the maximum
of the Markov widths of its blocks (maximal $2$-connected subgraphs). As every
$2$-connected $K_4$-minor-free graph is series-parallel (in case of a forest,
the blocks are single edges), the Markov width of every graph with no $K_4$ minor
is at most four by Theorem~\ref{thm-sp}.
\end{proof}

\section{Lower bound for complete graphs}

We will derive a lower bound on the Markov width of a complete graph
from bounds on maximum density of clean triangulations of surfaces.
A triangulation $T$ of a surface is {\em clean}
if every triangle of $T$ is a face triangle. A triangulation $T$ is
{\em $2$-face-colorable} if its faces can be colored with two colors
in such a way that every two adjacent faces receive distinct colors,
i.e., the dual graph of $T$ is bipartite. For other, more standard, definitions
related to triangulations of surfaces, we refer the reader
to the monograph~\cite{bib-mohar+}.

\begin{lemma}
\label{lm-lower}
If there exists a clean $2$-face-colorable triangulation $T$
with $n$ vertices and $m$ edges, then the Markov width $\mu(K_n)$
of the complete graph $K_n$ is at least $m/3$.
\end{lemma}

\begin{proof}
Assume that the faces of $T$ are colored with red and blue and
let $G=K_n$ be the complete graph on the same vertex set as $T$,
which we identify with the set $[n]$. Remember that $\Z([n])$ has
a basis $e_a$, where $a$ ranges over all functions from $[n]$ into $\{0,1\}$,
which could be considered as indicator functions of subsets of $[n]$.
Let $z^r \in \Z([n])$ be defined as the sum of vectors $e_{\chi(F)}$ over all red faces $F$ of $T$.
In particular, $||z^r||_1=m/3$. The vector $z^b\in\Z([n])$ is defined analogously with respect
to blue faces. It is not hard to see that $\pi_G(z^r)=\pi_G(z^b)$. Let $\Pi_G(z^r)=\{z \in \Z([n]) \; | \; \pi_G(z)=\pi_G(z_r)\}$
be the fiber of $\pi_G$ containing $z^r$. We will show that $\Pi_G(z^r)=\{z^r,z^b\}$ which will imply the statement
of the lemma by the definition of the Markov width.

Consider $z\in\Pi_G(z^r)$ and let $A_1,\ldots,A_{m/3}$ be (not necessarily distinct) subsets of $[n]$ such
that $$z_{\chi(A)}=\sum_{i=1}^{m/3}e_{\chi(A_i)}.$$
By considering $\pi_{ij}(z^r)=\pi_{ij}(z)$ for $i,j \in [n]$ we see that if $ij \in E(T)$
then $\{i,j\}$ belong to exactly one member of $(A_1,\ldots,A_{m/3})$ and, if $ij \not \in E(T)$ then $\{i,j\}$ belongs to no member of this family. It follows that no two sets $A_1$ share more than one elements, and that every $A_i$ forms a vertex set of a complete subgraph of $T$.  Since $T$ is a $2$-face-colorable clean triangulation, $T$ contains no subgraph isomorphic to $K_4$ and
consequently $|A_i|\le 3$. On the other hand, the sum of the sizes of $A_1,\ldots,A_{m/3}$ is independent on the choice of $z \in \Pi_G(z^r)$ and we have
$$\sum_{i=1}^{m/3}|A_i|=m.$$
This implies that $|A_1|=|A_2|=\ldots=|A_{m/3}|=3$.
Hence, each $A_i$ corresponds to a face of $T$. Therefore $A_1,\ldots,A_{m/3}$ is a collection of the vertex sets
pairwise non-adjacent faces of $T$, such that every edge of $T$ is contained in some face of the collection. There are only two such collections, namely the vertex sets of the red faces and of the blue faces of $T$.
In particular, $z\in\{z^r,z^b\}$, as claimed.
\end{proof}

We now apply Lemma~\ref{lm-lower} to obtain a lower bound
on $\mu(K_n)$. A simple triangulation satisfying the assumptions of Lemma~\ref{lm-lower}
is a double-wheel drawn on the sphere: It has $N+2$ vertices and
$3N$ edges. Hence, Lemma~\ref{lm-lower} gives us that the Markov
width of $K_n$ is at least $n-2$. To obtain a superlinear bound,
we need denser clean triangulations. Such triangulations of essentially optimal density were
constructed by Seress and Szab{\'o} in~\cite{bib-seress+}. For every $\varepsilon>0$ and for every sufficiently large integer $n$, they construct a clean triangulation $T_{n,\varepsilon}$ of some surface $\Sigma$
with $n$ vertices and $n^{2-\varepsilon}$ edges. The constructed triangulations are $3$-vertex-colorable.
If the corresponding surface $\Sigma$ is orientable then the triangulation
is $2$-face-colorable, as the clockwise orders of the vertex colors around adjacent faces are different.  If $\Sigma$ is non-orientable, it is possible to obtain a clean triangulation with $2n$ vertices and $2n^{2-\varepsilon}$ edges
which is $3$-vertex-colorable by considering an orientable $2$-cover of $\Sigma$ and a corresponding clean triangulation.

We can now infer from Lemma~\ref{lm-lower} the following:

\begin{theorem}
\label{thm-lower}
For every $\varepsilon>0$, there exists $n_0$ such that
$\mu(K_n)\ge n^{2-\varepsilon}$ for all $n\ge n_0$.
\end{theorem}

\section{Final Remarks}

\begin{enumerate}
\item Our definition of binary graph models differs from the standard one. We do not consider $1$-way marginals corresponding to the vertices of the graph. This is a very minor distinction, as for graphs of minimum degree one the fibers of $\pi_G$ remain unchanged, and a binary graph model of a disconnected graph can be considered as a toric fiber product (see~\cite{bib-sullivant}) of the models corresponding to its components.

\item In~\cite{bib-sturmfels+} Sturmfels and Sullivant consider cut ideals of graphs. Those are binomial ideals which are closely related to binary graph models as pointed out in~\cite{bib-sturmfels+}. In particular, the Markov width $\mu(G)$ of a graph $G$ equals to the maximum degree of a binomial appearing in a minimal generating set of the cut ideal $I_{\hat G}$, where $\hat G$ is a graph obtained from $G$ by adding a universal vertex, that is a new vertex joined by an edge to every vertex of $G$.

Sturmfels and Sullivant conjecture that $I_G$ can be generated in degree four if and only if $G$ has no $K_5$ minor. Our result can be interpreted as a partial result towards this conjecture, verifying it for all graphs containing a universal vertex. Also let us mention that Engstr{\" o}m~\cite{bib-engstrom} recently proved another conjecture from~\cite{bib-sturmfels+} showing that a cut ideal $I_G$ can be generated in degree two if and only if $G$ has no $K_4$ minor.
\end{enumerate}

\section*{Acknowledgements}

The second author would like to thank Bernd Sturmfels for referring him to~\cite{bib-develin03+} and for inspiring discussions.


\begin{thebibliography}{9}
\bibitem{bib-engstrom}
A. Engstr{\" o}m:
Cut ideals of $K_4$-minor free graphs are generated by quadrics,
manuscript.
\end{thebibliography}
\end{document}